\documentclass[11pt,oneside,a4paper]{amsart}
\usepackage{graphicx}
\usepackage{amsthm}
\usepackage{hyperref}
\usepackage{amsmath}
\usepackage{amssymb}
\usepackage[T1]{fontenc}
\usepackage{color}

\usepackage{mathrsfs}

\newtheorem{lem}{Lemma}[section]
\newtheorem{thm}[lem]{Theorem}
\newtheorem{rem}[lem]{Remark}

\newtheorem{cor}[lem]{Corollary}

\newtheorem{prop}[lem]{Proposition}

\def \NN{\mathbb{N}}

\def \RR{\mathbb{R}}
\def \Rd{{\RR^d}}

\newcommand{\ind}{{\bf 1 }}

\newcommand{\Kz}{K}

\title[Sharp Gaussian estimates]{Characterization of sharp global Gaussian estimates for Schr{\"o}dinger heat kernels}

 

\author{Krzysztof Bogdan}
 \address{Wroc{\l}aw University of Science and Technology,
Wybrze{\.z}e Wyspia{\'n}skiego 27,
50-370 Wroc{\l}aw, Poland}
\email{bogdan@pwr.edu.pl}

\author{Jacek Dziuba{\'n}ski}
\address{University of Wroc{\l}aw,
Pl. Grunwaldzki 2/4,
50-384 Wroc{\l}aw,
Poland}
\email{Jacek.Dziubanski@math.uni.wroc.pl}

\author{Karol Szczypkowski}
\address{Universit{\"a}t Bielefeld, Postfach 10 01 31,
D-33501 Bielefeld, Germany  and
Wroc{\l}aw University of Science and Technology,
Wybrze{\.z}e Wyspia{\'n}skiego 27,
50-370 Wroc{\l}aw, Poland
}
\email{karol.szczypkowski@math.uni-bielefeld.de, karol.szczypkowski@pwr.edu.pl}

\date{\today}
\thanks{Jacek Dziuba{\'n}ski was supported by the Polish National Science Center
(Narodowe Centrum Nauki) grant DEC-2012/05/B/ST1/00672.}  

\begin{document}

\begin{abstract}
We investigate when the fundamental solution of the Schr\"o\-dinger equation $\partial_t=\Delta+V$ posseses sharp Gaussian bounds global in space and time. We give a characterization for $V\leq 0$
and
a sufficient condition
for general $V$. 
\end{abstract}

\maketitle

\section{Introduction and main results}\label{sec:i}

Let $d\in\NN$. For $x,y\in\Rd$ and $t>0$ we consider the Gaussian kernel
$$g(t,x,y)=g(t,y-x)=(4\pi t)^{-d/2} e^{-|y-x|^2/(4t)}.$$  
It is the fundamental solution of  the heat equation $\partial_t=\Delta$. 
For a 
function $V\colon\Rd\to\RR$
we let $G$
be 
the fundamental solution of $\partial_t=\Delta+V$, determined by the following Duhamel or perturbation formula for $t>0$, $x,y\in \Rd$,
\[
G(t,x,y)=g(t,x,y)+\int_0^t \int_\Rd G(s,x,z)V(z)g(t-s,z,y)dzds.
\]
We aim at the {\it sharp global Gaussian bounds} of $G$, which mean that there are numbers $0< c_1\le 1 \le c_2$ such that
\begin{align}\label{est:sharp_uni}
c_1  \leq \frac{G(t,x,y)}{g(t,x,y)}\leq c_2\,, \quad t>0,\ x,y\in\Rd.
\end{align}
Clearly, \eqref{est:sharp_uni} implies the plain global Gaussian bounds, which only require numbers $0<\varepsilon_1, c_1 \leq 1\leq \varepsilon_2, c_2 <\infty$ such that for all $t>0$ and $x,y\in\Rd$,
\begin{align}\label{est:gaus}
c_1\, (4\pi t)^{-d/2} e^{-\frac{|y-x|^2}{4t\varepsilon_1}} \leq G(t,x,y)\leq c_2\, (4\pi t)^{-d/2} e^{-\frac{|y-x|^2}{4t\varepsilon_2}}.
\end{align}
To characterize \eqref{est:sharp_uni} we let
\begin{align*}
S(V,t,x,y)=\int_0^t \int_{\Rd} \frac{g(s,x,z)g(t-s,z,y)}{g(t,x,y)}|V(z)|\,dzds\,, \quad t>0,\ x,y\in \Rd.
\end{align*}
This will often be abbreviated to $S(V,t)$, $S(V)$ or $S$, and we always assume that $V$ is Borel measurable.
Denote, as usual,
\begin{equation*}\label{def:sSbi}
\|S(V) \|_{\infty}=\sup_{t>0,\,x,y\in\Rd} S(V,t,x,y).
\end{equation*}
The results of Bogdan, Hansen and Jakubowski \cite{MR2457489} and Zhang \cite{MR1978999}  give enough evidence in favor of using $S(V)$ 
in this and more general contexts.

We will say that $V$ has bounded potential for bridges globally in time, if $\|S(V) \|_{\infty}<\infty$, 
in which case
we can largely resolve \eqref{est:sharp_uni} thanks to the following 
folklore
result.
\begin{lem}\label{lem:gen_neg}
If $\eta:=\|S(V^+)\|_\infty<1$
and $S(V^-)$ is locally bounded,
then
\begin{align}\label{gen_est}
e^{-S(V^-,t,x,y)} \leq \frac{G(t,x,y)}{g(t,x,y)}\leq \frac{1}{1-\eta}, \qquad t>0, \ x,y\in \Rd \,.
\end{align}
If $V\leq 0$, then \eqref{est:sharp_uni} holds if and only if $\|S(V)\|_\infty<\infty$.
If $V\geq 0$, then \eqref{est:sharp_uni} implies $\|S(V)\|_\infty<\infty$.
\end{lem}
Here, as usual, $V^+=\max(0,V)$ and $V^-=\max(0,-V)$.
The last statement of the lemma easily follows from Duhamel formula.
The rest of the lemma is an excerpt from \cite[Lemma~1.1 and Lemma~1.2]{2015arXiv151107167B}, where it is proved based on
\cite{MR2457489, MR3000465}.
We note that $S(V)=\infty$ 
for every nontrivial $V$ in dimensions
$d = 1$ and $2$,
see, e.g., \cite[Lemma~1.3]{2015arXiv151107167B}, and so \eqref{est:sharp_uni} is impossible for nontrivial $V\ge 0$ and nontrivial $V\le 0$ in these dimensions.
To characterize the boundedness of
$S(V)$,
for $d\geq 3$ and $x,y\in \Rd$ we define
\begin{align}\label{Kernel}
\Kz (x,y)= \frac{e^{-
\left(|x||y|-x\cdot y \right)/2}}{|x|^{d-2}} \left(1+|x||y| \right)^{d/2-3/2}\,,
\end{align}
where $x\cdot y$ is the usual scalar product, and we let
\begin{eqnarray*}
\Kz  (V,x,y)&=&\int_{\Rd} |V(z)|\Kz (z-x,y)\,dz\,.
\end{eqnarray*}
We also denote 
$$\|V \|_{\Kz }= \|
\Kz (V)\|_\infty\,.$$ 
Here is our main result.
\begin{thm}\label{thm:J_0}
There are constants 
$M_1$, $M_2$  depending only on $d$, such that
\begin{equation}\label{eq:cSK0}
M_1 \| V \|_{\Kz } \leq 
\|S(V)\|_{\infty}
\leq M_2 \| V \|_{\Kz }\,.
\end{equation}
\end{thm}
Here by constants we mean positive numbers. The proof of Theorem~\ref{thm:J_0} is given in Section~\ref{sec:proofs}.
In view of \eqref{eq:cSK0} and of the second and the third statements of Lemma~\ref{lem:gen_neg}, the condition $\|V\|_{\Kz}<\infty$ may replace $\|S(V)\|_{\infty}<\infty$ in characterizing \eqref{est:sharp_uni}, which will be often used without mention. 
Similarly, sufficient smallness of $\|V\|_K$ yields \eqref{gen_est} in view of the first statement of Lemma~\ref{lem:gen_neg}.

\begin{cor}\label{cor:sGbbyK}
If $V\leq 0$, then \eqref{est:sharp_uni} holds if and only if $\Kz (V)$ is bounded.
\end{cor}
Compared 
with
$S(V)$, 
$\Kz (V)$ is easier to investigate, because $K(V)$ has one argument less than $S(V)$. This
leads to considerable progress in analysis of \eqref{est:sharp_uni}, which we now present. 
For $d\geq 3$ we let $C_d=\Gamma(d/2-1)/(4\pi^{d/2})$ and
\begin{align*}
-\Delta^{-1} V(x)
= \int_0^{\infty}\int_{\Rd}g(t,x,z) V(z)\,dzdu
=C_d\int_{\Rd} \frac{1}{|z-x|^{d-2}}V(z)\,dz\,.
\end{align*} 
For $d=3$ the formula for $\Kz $ simplifies and we easily obtain
\begin{align}\label{eq:d_3}
\|V\|_{\Kz }= C_{d}^{-1}\, \|\Delta^{-1} |V| \|_{\infty}\,,
\end{align}
thus 
$\|\Delta^{-1} |V|\|_{\infty}$
resolves \eqref{est:sharp_uni} in the same way as $\|S(V)\|_{\infty}$.
For instance, if $d=3$ and $V\leq 0$, then
the sharp global Gaussian bounds \eqref{est:sharp_uni} are equivalent to the condition 
$\|\Delta^{-1} V \|_{\infty}<\infty$.
This equivalence was first proved by Milman and Semenov \cite[Remark~(3) on p. 4]{MR1994762}. 

The main focus of the present paper is on the case of $d\ge 4$. Let 
$$\|V \|_{d/2}=\left(\int_\Rd |V(z)|^{d/2}dz\right)^{2/d}.$$
\begin{prop}\label{thm:D_Ld/2}
If $d\geq 4$, then 
\begin{equation}\label{eq:KV}
C_d^{-1}\|\Delta^{-1}|V|\|_{\infty}\le
\|V\|_{\Kz } \leq 2^{(d-3)/2}\Big( C_d^{-1}\|\Delta^{-1}|V|  \|_{\infty} + \kappa_d \|V \|_{d/2}\Big).
\end{equation}
\end{prop} 
The result is an analogue of \cite[Corollary~1]{MR1642818}.
In Section~\ref{sec:proofs} we give the proof and specify the constant $\kappa_d$.
As a consequence,
$\|\Delta^{-1}V\|_{\infty}<\infty$ is necessary for 
\eqref{est:sharp_uni} if $V\le 0$ and if $V\ge 0$, cf. Lemma~\ref{lem:gen_neg}.
On the other hand for every $d\geq 3$
there is $V\leq 0$ such that 
$\|V\|_{\Kz}<\infty$, i.e., \eqref{est:sharp_uni} holds,
but
$V\notin L^1(\Rd)\cup \bigcup_{p>1}L^p_{loc}(\Rd)$, in particular $\|V \|_{d/2}=\infty$, see 
\cite{2015arXiv151107167B}. 

A long-standing open problem on \eqref{est:sharp_uni} for $V\le 0$ posed by Liskevich and Semenov \cite[p. 602]{MR1642818} reads as follows: ``The validity of the two-sided estimates for the case $d>3$ without the additional assumption $V\in L^{d/2}$ is an open question.'' 
The question is whether $\|V\|_{\Kz}$ and $\|\Delta^{-1}V\|_{\infty}$ are comparable for $d>3$.
It turns out that the answer is negative, as follows.
\begin{prop}\label{thm:dist}
Let $d\geq 4$. 
For $\mathbf z=(z_1, z_2,\ldots,z_d)\in\Rd$ we write $\mathbf z=(z_1,\mathbf z_2)$, 
where $\mathbf z_2=(z_2,\ldots,z_d)\in \RR^{d-1}$. 
We define
\begin{eqnarray*}
A&=&\{(z_1,\mathbf z_2)\in \Rd : \ z_1>4, \ |\mathbf z_2|\leq \sqrt{z_1}\}, \quad \mbox{ and }\\
V(z_1,\mathbf z_2)&=&-\frac{1}{z_1} \ind_{A}(z_1,\mathbf z_2).
\end{eqnarray*} 
Then $\| \Delta^{-1} V \|_{\infty}<\infty$ and 
$\|V\|_{\Kz }=\infty$.
There even is function $V \le  0$ with compact support and such that $\|\Delta^{-1}V\|_{\infty}<\infty$ but $\|V\|_{\Kz }=\infty$.
\end{prop}
Generally,
for $d\ge 4$,
neither finiteness nor smallness of $\|\Delta^{-1} V \|_{\infty}$
are sufficient for the
comparability of $g$ and $G$, even for $V$ with fixed sign and compact support.

Here are a few more comments to relate our result to existing literature.
In \cite{MR2064932} Milman and Semenov denote
$e(V,0)=\|\Delta^{-1}|V|\|_{\infty}$ and introduce
$e_*(V,0)=\sup_{\alpha\in\Rd}\|V (-\Delta+2\alpha\cdot\nabla)^{-1}\|_{1\to 1}$ to describe \eqref{est:sharp_uni} -- see \cite[Theorem~1C]{MR2064932}.
The spatial anisotropy introduced by $\alpha\cdot\nabla$ has a similar role as 
that seen in the integral defining $S(V,t,x,y)$ and there are constants $c_1$, $c_2$ depending only on $d\ge 3$ such that
$$
c_1 \|V\|_{\Kz } \leq e_*(V,0)\leq c_2 \|V\|_{\Kz }\,.
$$
This is proved in \eqref{last_step} below. For $d=3$ we have $e(V,0)=e_*(V,0)$. On the contrary,  for $d\geq 4$ by Proposition~\ref{thm:dist} there is $V\leq 0$ such that $e(V,0)<\infty$ but $e_*(V,0)=\infty$.

For the last remark we restrict ourselves to $V\le 0$.
Then
the condition $
\|\Delta^{-1}V\|_{\infty}<\infty$ characterizes 
the plain global Gaussian bounds \eqref{est:gaus}, see \cite{MR1456565}.
By \eqref{eq:d_3}, for $d=3$ (and $V\leq 0$) 
the plain global Gaussian bounds \eqref{est:gaus} hold if and only if the sharp global Gaussian bounds \eqref{est:sharp_uni} hold.
In contrast, by Proposition~\ref{thm:dist} 
for 
$d\geq 4$ the property \eqref{est:gaus}  
is weaker than \eqref{est:sharp_uni}.

The remainder of the paper is structured as follows.
In Section~\ref{sec:proofs} we prove 
Theorem~\ref{thm:J_0},
Proposition~\ref{thm:D_Ld/2}
and
Proposition~\ref{thm:dist}.
Section~\ref{appendix}
gives auxiliary results,
in particular the following crucial estimate of an
inverse-Gaussian type integral.
\begin{thm}\label{thm:est2}
Let $c>0$, $\beta> 1$ and
\begin{align*}
f(a,b)&=\int_0^{\infty} u^{-\beta} e^{-c \left[ \sqrt{u}b - \frac{a}{\sqrt{u}} \right]^2}\,du\,,\qquad a,b>0\,.
\end{align*}
We have
$$
f(a,b)\overset{\beta,c}{\approx} \frac{(1+4ab)^{\beta-3/2}}{a^{2(\beta-1)}}\,.
$$
\end{thm}
Here $\overset{\beta,c}{\approx}$ means that the ratio of both sides is bounded above and below by constants depending only on $\beta$ and $c$.

\section{Proofs of main results}\label{sec:proofs}

For $t>0$, $x,y\in \Rd$, we consider
\begin{align}
N(V,t,x,y)&:=\int_0^{t/2}\int_{\Rd}\frac{e^{-|z-y+(\tau/t)(y-x)|^2/(4\tau)}}{\tau^{d/2}}|V(z)|dzd\tau \nonumber\\
&+\int_{t/2}^t\int_{\Rd} \frac{e^{-|z-y+(\tau/t)(y-x)|^2/(4(t-\tau))}}{(t-\tau)^{d/2}} |V(z)|dzd\tau= N(V,t,y,x) \,.\label{def:N}
\end{align}
Clearly, $S(V)=S(|V|)$ and $N(V)=N(|V|)$. 
Because of the work of Zhang \cite{MR1978999}, 
$N$ is a  convenient approximation of $S$.
Namely, by \cite[Lemma 3.1, Lemma 3.2]{MR1978999}, there  are constants $m_1,m_2$ depending only on $d$ such that
\begin{align}
S(V,t,x,y)&\geq m_1\, N(V,t/2,x,y)\,,\qquad t>0, \ x,y\in \Rd \,,\tag{L} \label{L}\\
S(V,t,x,y)&\leq m_2\, N(V,t,x,y)\,,\qquad t>0, \ x,y\in \Rd \,.\tag{U} \label{U}
\end{align}
As seen in \cite{MR1978999}, the comparability even holds for the kernels of $S$ and $N$.

In this section we prove out main result, i.e., Theorem~\ref{thm:J_0}. We start by using $N(V,t)$, \eqref{U} and \eqref{L}, to estimate $S(V,t)$.
\begin{lem}\label{lem:upr}
Let $t>0$. We have
\begin{align*}
\int_0^{t/2}\int_{\Rd}\frac{e^{-|z-y+(\tau/t)(y-x)|^2/(4\tau)}}{\tau^{d/2}}|V(z)|\,dzd\tau  \leq N(V,t)(x,y)\,,\quad x,y\in \Rd\,,
\end{align*}
and
\begin{align*}
\sup_{x,y}N(V,t)(x,y) \leq 2 \,\sup_{x,y} \int_0^{t/2}\int_{\Rd}\frac{e^{-|z-y+(\tau/t)(y-x)|^2/(4\tau)}}{\tau^{d/2}}|V(z)|\,dzd\tau\,.
\end{align*}
\end{lem}
\begin{proof}
The first inequality follows by the definition of $N(V,t)(x,y)$. For the proof of the second one we note that
\begin{align*}
\int_{t/2}^t\int_{\Rd} &\frac{e^{-|z-y+(\tau/t)(y-x)|^2/(4(t-\tau))}}{(t-\tau)^{d/2}} |V(z)|dzd\tau\\
&\qquad = \int_0^{t/2}\int_{\Rd} \frac{e^{-|z-x+(\tau/t)(x-y)|^2/(4\tau)}}{\tau} |V(z)|dzd\tau\,.
\end{align*}
\end{proof}
For $x,y\in\Rd$ we let
$$J(x,y)=\int_0^{\infty} \tau^{-d/2} e^{-\frac{|x-\tau y|^2}{4\tau}} d\tau\,.$$
In view of the discussion in Section~\ref{sec:i} we have
\begin{align*}
e_*(V,0)&=\sup_{\alpha \in \Rd}\|(-\Delta+2\alpha\cdot\nabla )^{-1}|V|\|_{\infty}\\
&=(4\pi)^{-d/2}\sup_{x,y\in\Rd} \int_{\Rd} J(z-x,y) |V(z)|\,dz\,.
\end{align*}
\begin{lem}\label{lem:J_0}
We have
\begin{align*}
\sup_{t>0,\,x,y\in\Rd} S(V,t,x,y) \geq m_1\sup_{x,y\in\Rd} \int_{\Rd} J(z-x,y)|V(z)|\,dz\,.
\end{align*}
and
\begin{align*}
\sup_{t>0,\,x,y\in\Rd} S(V,t,x,y)\leq 2\, m_2 \sup_{x,y\in\Rd} \int_{\Rd} J(z-x,y)|V(z)|\,dz\,,
\end{align*}
\end{lem}

\begin{proof}
By \eqref{L} and Lemma~\ref{lem:upr},
\begin{align*}
\sup_{t>0,\,x,y\in\Rd}& S(V,t,x,y)\geq m_1 \sup_{t>0,\,x,y\in\Rd} N(|V|,t/2)(x,y)\\
&\geq m_1 \sup_{t>0,\,x,y\in\Rd} \int_0^{t/4}\int_{\Rd}\frac{e^{-|z-y+(2\tau/t)(y-x)|^2/(4\tau)}}{\tau^{d/2}}|V(z)|\,dzd\tau\\
&= m_1\sup_{t>0,\,y,w\in\Rd} \int_0^{t/4}\int_{\Rd}\frac{e^{-|z-y+\tau w|^2/(4\tau)}}{\tau^{d/2}}|V(z)|\,dzd\tau\\
&= m_1 \sup_{y,w\in\Rd} \int_{\Rd} J(z-y,w) |V(z)|\,dz\,.
\end{align*}
By \eqref{U} and Lemma~\ref{lem:upr},
\begin{align*}
\sup_{t>0,\,x,y\in\Rd}& S(V,t,x,y)\leq m_2\sup_{t>0,\, x,y\in\Rd} N(V,t)(x,y)\\
&\leq 2\, m_2 \sup_{t>0,\, x,y\in\Rd} \int_0^{t/2}\int_{\Rd}\frac{e^{-|z-y+(\tau/t)(y-x)|^2/(4\tau)}}{\tau^{d/2}}|V(z)|\,dzd\tau\\
&\leq 2\, m_2 \sup_{t>0,\, y,w\in\Rd} \int_0^{t/2}\int_{\Rd}\frac{e^{-|z-y+\tau w|^2/(4\tau)}}{\tau^{d/2}}|V(z)|\,dzd\tau\\
&= 2 \, m_2 \sup_{y,w\in\Rd} \int_{\Rd} J (z-y,w) |V(z)|\,dz\,.
\end{align*}
\end{proof}

\begin{proof}[Proof of Theorem~\ref{thm:J_0}]
We claim \eqref{eq:cSK0} holds
with $M_1>0$ that depends only on $d$, and
$M_2= m_2 2^{d} \int_0^{\infty} (1\vee r)^{d/2-3/2}\, r^{-1/2}\, e^{-r}\,dr$.
To this end, according to Lemma~\ref{lem:J_0},  we analyze
\begin{align*}
J(z-x,y)=
\int_0^{\infty} \tau^{-d/2} e^{-\frac{|z-x-\tau y|^2}{4\tau}} d\tau\,.
\end{align*}
Obviously, $J=\infty$ if $d=1$ or $d=2$. For $d\geq 3$ we observe that
$$
\frac{|z-x-\tau y|^2}{4\tau}=\frac1{4}\left[\frac{|z-x|}{\sqrt{\tau}}-\sqrt{\tau}|y| \right]^2+\frac1{2}\big(|z-x||y|-(z-x)\cdot y\big)\,,
$$
and thus
$$
J(z-x,y)= e^{-\frac1{2}\left(|z-x||y|-(z-x)\cdot y \right)}
\int_0^{\infty} \tau^{-d/2} 
e^{-\frac1{4}\left[\sqrt{\tau}|y|-\frac{|z-x|}{\sqrt{\tau}} \right]^2} d\tau\,.
$$
Finally, by Theorem~\ref{thm:est2} 
with $a=|z-x|/2$, $b=|y|/2$, $\beta=d/2$ and $c=1$,
\begin{align}\label{last_step}
J(z-x,y)\overset{d}{\approx} 
\Kz (z-x,y)\,.
\end{align}
This also gives the explicit constants, as a consequence of Remark~\ref{rem:expl_final}. For instance we can take $M_2= 8m_2 \sqrt{\pi}$ if $d=3$.
\end{proof}

\begin{proof}[Proof of Proposition~\ref{thm:D_Ld/2}]
The left hand side inequality follows from the identity
$\Kz  (V)(x,0)= C_d^{-1} (-\Delta^{-1})|V|(x)$.
If $y=0$, then the upper bound trivially holds. For $y\neq 0$
we consider two domains of integration. We have
\begin{align*}
\int_{|z-x||y|\leq 1} &\Kz (z-x,y)|V(z)|\,dz
\leq 2^{(d-3)/2} \int_{|z-x||y|\leq 1} \frac{1}{|z-x|^{d-2}} |V(z)|dz\\
&\leq \frac{2^{(d-3)/2}}{C_d} |\!|\Delta^{-1}|V| |\!|_{\infty}\,.
\end{align*}
Furthermore, by a change of variables and H{\"o}lder inequality,
\begin{align*}
&\int_{|z-x||y|\geq 1} \Kz (z-x,y)|V(z)|dz\\
&\leq 2^{(d-3)/2} \!\!\! \int_{|z-x||y|\geq 1}
\frac{e^{-\frac1{2}\left(|z-x||y|-(z-x)\cdot y \right)}}{(|z-x||y|)^{(d-1)/2}} |y|^{d-2}|V(z)|dz
\leq 2^{(d-3)/2}\kappa_d  |\!|V |\!|_{d/2}
\,,
\end{align*}
where
$$
\kappa_d=\left(
\int_{|w|>1} 
\left(e^{-\frac1{2}(|w|-w\cdot 1 )}|w|^{-(d-1)/2}\right)^{d/(d-2)} dw \right)^{(d-2)/d} \,.
$$
The finiteness of $\kappa_d$
follows from Lemma~\ref{lem:const} below.
\end{proof}

\begin{proof}[Proof of Proposition~\ref{thm:dist}]
We use the notation introduced in the formulation of the theorem.
First we prove that $\|V\|_{\Kz }=\infty$.
Let $\mathbf  y=(1,\mathbf 0)\in \mathbb R^d$, $\mathbf x=0$.
Observe that for $\mathbf z\in A$ we have
$$ 0\leq |\mathbf z||\mathbf y|-\mathbf z\cdot \mathbf y=|\mathbf z|-z_1 = \frac{|\mathbf z_2|^2}{\sqrt{z_1^2+|\mathbf z_2|^2}+z_1}\leq \frac{z_1}{\sqrt{z_1^2+|\mathbf z_2|^2}+z_1}\leq 1 $$
and thus also
$ z_1\le |\mathbf z| \le 2z_1$. Then,
\begin{align*}
\|V\|_{\Kz }&\geq \int_{\mathbb R^d} e^{-\frac{1}{2}(|\mathbf z||\mathbf y|-\mathbf z\cdot \mathbf y)}|V(\mathbf z)|\frac{(1+|\mathbf z||\mathbf y|)^{\frac{d}{2}-\frac{3}{2}}}{|\mathbf z|^{d-2}}\, d\mathbf z
\geq c\int_A \frac{1}{z_1}\frac{z_1^{\frac{d}{2}-\frac{3}{2}}}{z_1^{d-2}}\, d\mathbf z\\
&=c\int_{4}^\infty \int_{|\mathbf z_2|<\sqrt{z_1}} z_1^{-1+2-d+\frac{d}{2}-\frac{3}{2}}\, d\mathbf z_2\, dz_1
=c_1\int_{4}^\infty  z_1^{-1+2-d+\frac{d}{2}-\frac{3}{2}+\frac{1}{2}(d-1)}\, dz_1\\
&=c_1\int_{4}^\infty \frac{1}{z_1}\, dz_1=\infty. 
\end{align*}
We now prove that $\|\Delta^{-1} |V| \|_{\infty}<\infty$.
By the symmetric rearrangement inequality (see \cite[Chapter~3]{MR1817225}) we have
\begin{align*}
\sup_{\mathbf x\in \mathbb R^d}\int_{\mathbb R^d} \frac{1}{|\mathbf z-\mathbf x|^{d-2}} |V(\mathbf z)|\, d\mathbf z
= \sup_{x_1\in \mathbb R} \int_4^{\infty} \int_{\RR^{d-1}} \frac{\ind_{|\mathbf z_2|<\sqrt{z_1}}}{[(z_1-x_1)^2+|\mathbf z_2|^2\,]^{(d-2)/2}}\frac1{z_1}d\mathbf z_2\,dz_1
\end{align*}
It suffices then to consider $\mathbf x=(x_1,0,\ldots, 0)$ and
we only need to show that the following three integrals are uniformly bounded for $x_1\geq 4$.
The first integral is
\begin{align*}
I_1&=\int_{x_1+\sqrt{x_1}}^\infty  \int_{|\mathbf z_2|<\sqrt{z_1}}
\frac{1}{|z_1-x_1|^{d-2}+|\mathbf z_2|^{d-2}} \frac{1}{z_1}\, d\mathbf z_2\, dz_1\\ 
& \leq \int_{x_1+\sqrt{x_1}}^\infty  \int_{|\mathbf z_2|<\sqrt{z_1}}
\frac{1}{|z_1-x_1|^{d-2}} \frac{1}{z_1}\, d\mathbf z_2\, dz_1\\
&= c \int_{x_1+\sqrt{x_1}}^\infty  
\frac{1}{|z_1-x_1|^{d-2}} \frac{1}{z_1} z_1^{\frac{1}{2}(d-1)}\, dz_1
 = c  \int_{\sqrt{x_1}}^\infty
\frac{1}{z_1^{d-2}} \ (z_1+x_1)^{\frac{d}{2}-\frac{3}{2}}\, dz_1\\
&\leq c' \int_{\sqrt{x_1}}^{x_1}
\frac{1}{z_1^{d-2}} \ x_1^{\frac{d}{2}-\frac{3}{2}}\, dz_1 
+c' \int_{x_1}^\infty 
\frac{1}{z_1^{d-2}} \ z_1^{\frac{d}{2}-\frac{3}{2}}\, dz_1 \leq c''<\infty.
\end{align*}
The second integral we consider is
\begin{align*}
I_2 & = \int_2^{x_1-\sqrt{x_1}}  \int_{|\mathbf z_2|<\sqrt{z_1}}
\frac{1}{{|z_1-x_1|^{d-2}+|\mathbf z_2|^{d-2}}} \frac{1}{z_1}\, d\mathbf z_2\, dz_1\\
& \leq
\int_2^{x_1-\sqrt{x_1}}  \int_{|\mathbf z_2|<\sqrt{z_1}}
\frac{1}{{|z_1-x_1|^{d-2}}} \frac{1}{z_1}\, d\mathbf z_2\, dz_1
=c \int_2^{x_1-\sqrt{x_1}}  
\frac{1}{{(x_1-z_1)^{d-2}}} z_1^{\frac{d}{2} -\frac{3}{2}}\,  dz_1\\
&= c \int_{\sqrt{x_1}}^{x_1-2} \frac{1}{w^{d-2}} (x_1-w)^{\frac{d}{2}-\frac{3}{2}}\, dw
\leq c \int_{\sqrt{x_1}}^{x_1-2} \frac{1}{w^{d-2}} x_1^{\frac{d}{2}-\frac{3}{2}}\, dw
\le c'<\infty.
\end{align*}
The remaining integral is
\begin{align*}
I_3
&=\int_{x_1-\sqrt{x_1}}^{x_1+\sqrt{x_1}}  \int_{|\mathbf z_2|<\sqrt{z_1}}
\frac{1}{{|\mathbf z-\mathbf x|^{d-2}} }\frac{1}{z_1}\, d\mathbf z
\leq  2\int_{x_1-\sqrt{x_1}}^{x_1+\sqrt{x_1}}  \int_{|\mathbf z_2|<2\sqrt{x_1}}
\frac{1}{{|\mathbf z-\mathbf x|^{d-2}}} \frac{1}{x_1}\, d\mathbf z\\
&\leq 2 \int_{B(\mathbf x,\, 3 \sqrt{x_1})}\frac{1}{{|\mathbf z-\mathbf x|^{d-2}}} \frac{1}{x_1}\, d\mathbf z
\leq c< \infty.
\end{align*} 
To prove the second statement of Proposition~\ref{thm:dist}, for $s>0$ we let 
$
{\it d}_s 
f(x)=sf(\sqrt{s}x)$. Note that the dilatation does not change the norms:
$$
\| \Delta^{-1}({\it d}_s f)  \|_{\infty}=\| \Delta^{-1} f \|_{\infty}\,, \qquad \|{\it d}_s f\|_{\Kz }=\|f\|_{\Kz } \,. 
$$
Moreover, ${\rm supp} ({\it d}_s f) \subseteq B(0,r/\sqrt{s})$ if ${\rm supp} (f)\subseteq B(0,r)$, $r>0$.
Now, let $V\geq 0$ be a potential such that 
$\|\Delta^{-1}V\|_{\infty}=C<\infty$ and $\|V\|_{\Kz }=\infty$. 
Then for any $r>0$ we have $\|\Delta^{-1} (V\ind_{B_r})\|_{\infty}\leq C$ and $\|V\ind_{B_r}\|_{\Kz }\to \infty$ as $r\to \infty$.
For $n\in \NN$ we define $$V_n={\it d}_{r_n^2}(V\ind_{B_{r_n}}) \,,$$ 
 where $r_n$ is chosen such that 
$
\|V\ind_{B_{r_n}}\|_{\Kz }\geq 4^n
$.
Thus ${\rm supp}(V_n)\subseteq B(0,1)$.
Finally we consider $\tilde{V}=\sum_{n=1}^{\infty}V_n/2^n$. Then,
$$
\|\tilde{V}\|_{\Kz }\geq \|V_n\|_{\Kz }/2^n\geq 2^n \to \infty\,,\quad n \to \infty \,,
$$
and
$$
\|\Delta^{-1}\tilde{V}\|_{\infty}\leq \sum_{n=1}^{\infty} \|\Delta^{-1}V_n\|_{\infty}/2^n \leq C\,.
$$
\end{proof}

\section{Appendix}
\label{appendix}
In this section we collect 
auxiliary calculations.

\begin{lem}\label{lem:h}
Let $\gamma>-1/2$. Then
$$
h(x)=\int_0^{\infty} \left(x+s^2 \right)^{\gamma} e^{-cs^2}ds\,  \overset{\gamma,c}{\approx}\, \left(1+x\right)^{\gamma}\,,\qquad x\geq 0\,.
$$
\end{lem}
\begin{proof}
By putting $r=s^2$ we get 
\begin{align*}
h(x)=\left(1+x\right)^{\gamma} \int_0^{\infty} \left( \frac{x+r}{1+x} \right)^{\gamma} r^{-1/2}\,\frac{e^{-cr}dr}{2}\,,
\end{align*}
Since for all $x,r\geq 0$ we have
\begin{align*}
1\vee r \geq \frac{x}{1+x} + \frac{r}{1+x} \geq \begin{cases} r/2\,, \quad {\rm for}\quad x\in(0,1)\,,\\ 1/2\,, \quad {\rm for} \quad x \geq 1\,,\end{cases}
\end{align*}
the last integral in the above is comparable with a positive constant depending only on $\gamma$ and $c$.
\end{proof}

\begin{rem}\label{rem:expl}
\rm
If $\gamma\geq 0$, then
$h(x)  \leq \, C \left(1+x\right)^{\gamma}$, $x\geq 0$,
where $C=\frac1{2}\int_0^{\infty} (1\vee r)^{\gamma} r^{-1/2} e^{-cr}dr $.
\end{rem}

\begin{lem}\label{lem:Iapp}
Let $c>0$, $\beta> 1$ and
\begin{align*}
 I_{\rm app}(a,b)=\int_0^{\infty} 
\left(\frac{s+\sqrt{4ab+s^2}}{2a}  \right)^{2(\beta-1)}
\frac{e^{-cs^2}\,ds}{\sqrt{4ab+s^2}}\,,\qquad a,b>0\,.
\end{align*}
Then
\begin{align*}
 I_{\rm app}(a,b) \overset{\beta,c}{\approx} \frac{\left(1+ 4ab \right)^{\beta -3/2}}{a^{2(\beta-1)}}\,.
\end{align*}
\end{lem}
\begin{proof}
Observe that $0\leq s \leq \sqrt{4ab+s^2}$. Thus with $h(x)$ and $\gamma=\beta-3/2$ from Lemma~\ref{lem:h} we have
\begin{align}\label{expl_1}
2^{-2(\beta-1)} a^{-2(\beta-1)} \leq \frac{I_{\rm app}(a,b)}{h(4ab)}\leq a^{-2(\beta-1)}\,.
\end{align}
The assertion follows by Lemma~\ref{lem:h}.
\end{proof}

\begin{proof}[Proof of Theorem~\ref{thm:est2}]
By substitution $u=(a/b)r$ we obtain
$$
f(a,b)=(a/b)^{1-\beta}\int_0^{\infty} r^{-\beta+1} e^{-c ab\left[\sqrt{r}-\frac1{\sqrt{r}} \right]^2} \frac{dr}{r}\,.
$$
By change of variables from $r$ to $1/r$ we get
$$
f(a,b)=(a/b)^{1-\beta}\int_0^{\infty} r^{\beta-1} e^{-c ab\left[\sqrt{r}-\frac1{\sqrt{r}} \right]^2} \frac{dr}{r}\,.
$$
Finally, we let $ \sqrt{r} - 1/\sqrt{r}=s/\sqrt{ab}$, then $\left( \sqrt{r} -s/\sqrt{4ab} \right)^2= 1+s^2/(4ab)$. Note that $\sqrt{r}>s/\sqrt{4ab}$, hence $$r=\left( s/\sqrt{4ab}+\sqrt{1+s^2/(4ab)}\right)^2=\left( s+\sqrt{4ab+s^2} \right)^2 /(4ab)\,,$$
and
\begin{align*}
dr&= 2 \left( s+\sqrt{4ab+s^2} \right) \left( 1+s/\sqrt{4ab+s^2}\right)\,ds /(4ab)\\
&= 2 \left( s+\sqrt{4ab+s^2} \right)^2 \,ds/\left(4ab\sqrt{4ab+s^2}\right)\\
&= 2 r\, ds/\sqrt{4ab+s^2}\,.
\end{align*}
This gives
\begin{align*}
&f(a,b)=2 \int_{-\infty}^{\infty} \left(\frac{s+\sqrt{4ab+s^2}}{2a} \right)^{2(\beta-1)} \frac{e^{-cs^2}ds}{\sqrt{4ab+s^2}}\,.
\end{align*}
By splitting the last integral we have
\begin{align*}
f(a,b)=2 \int_0^{\infty} \left(\frac{s+\sqrt{4ab+s^2}}{2a} \right)^{2(\beta-1)} \frac{e^{-cs^2}ds}{\sqrt{4ab+s^2}}\\
+2 \int_0^{\infty} \left(\frac{-s+\sqrt{4ab+s^2}}{2a} \right)^{2(\beta-1)} \frac{e^{-cs^2}ds}{\sqrt{4ab+s^2}}\,.
\end{align*}
Since $\beta>1$ and $0\leq -s +\sqrt{4ab+s^2}\leq s+\sqrt{4ab+s^2}$, we have
\begin{align}\label{expl_2}
2\, I_{\rm app}(a,b)\leq f(a,b)\leq 4\, I_{\rm app}(a,b)\,.
\end{align}
The proof is ended by an applications of Lemma \ref{lem:Iapp}.
\end{proof}

\begin{rem}\label{rem:expl_final}
\rm
Using \eqref{expl_2}, \eqref{expl_1} and Remark~\ref{rem:expl}
we get an explicit constant in the upper bound in Theorem~\ref{thm:est2} for $\beta\geq 3/2$:
$$
f(a,b)\leq C  \,\frac{(1+4ab)^{\beta-3/2}}{a^{2(\beta-1)}}\,,
$$
where 
$$
C=2 \int_0^{\infty} (1\vee r)^{\beta-3/2}\, r^{-1/2}\, e^{-cr}\,dr\,.
$$
In particular if $\beta=3/2$, then $C= \sqrt{4\pi/c}$.
\end{rem}

We now verify the finiteness of $\kappa_d$ from the statement of Proposition~\ref{thm:D_Ld/2}.
\begin{lem}\label{lem:const}
Let $d\geq 3$. Then,
\begin{align*}
\int_{\Rd\setminus B(0,1)} e^{-(|w|-w\cdot 1)}|w|^{-\beta}dw <\infty\quad \iff \quad \beta>(d+1)/2\,.
\end{align*}
\end{lem}
\begin{proof}
We always have
$$\int_{\{w\in \Rd\setminus B(0,1): \ w\cdot 1\le |w|\sqrt{2}/2\}} e^{-(|w|-w\cdot 1)}|w|^{-\beta}dw<\infty,$$
therefore we only need to characterize the finiteness of the complementary integral.
We will follow the usual notation for spherical coordinates in $\Rd$ \cite{MR1530579}. In particular, $w\cdot 1=r\cos \varphi_1$ and the Jakobian is $r^{d-1}\prod_{k=1}^{d-2} \sin^k(\varphi_{d-1-k})$. We denote $\varphi=\varphi_1$, and we consider
\begin{align*}
&\int_1^{\infty} \int_0^{\pi/4} e^{-r(1-\cos \varphi)} r^{-\beta+d-1}\sin^{d-2}\varphi \,d\varphi\, dr\\
&=\int_0^{\pi/4} h(\varphi) \frac{\sin^{d-2}\varphi }{(1-\cos \varphi)^{d-\beta}}\, d\varphi\,,
\end{align*}
where $h(\varphi)=\int_{1-\cos \varphi}^{\infty} e^{-s} s^{-\beta+d-1}ds$.
If $\beta=d$, then $h(\varphi)\approx 1+|\log(1-\cos\varphi)|$ and $\int_0^{\pi/4}h(\varphi)\sin^{d-2}\varphi d\varphi<\infty$, as needed.
If $\beta>d$, then $h(\varphi)\approx (1-\cos \varphi)^{d-\beta}$, and the integral is finite, too.
If $\beta<d$, then $h(\varphi)\approx 1$ and $\int_0^{\pi/4} \frac{\sin^{d-2}\varphi }{(1-\cos \varphi)^{d-\beta}} d\varphi\approx \int_0^{\pi/4} \varphi^{(d-2)-2(d-\beta)}d\varphi$, which converges if and only if $\beta>(d+1)/2$.
\end{proof}

\newpage
\bibliographystyle{abbrv}
\bibliography{csG}
\end{document}